\documentclass[a4paper]{amsart}
\usepackage{amssymb}
\usepackage{amsmath}
\usepackage{xypic}
\newenvironment{proofof}{\textit{Proof of Theorem \ref{sec:introduction-1}.} }{ ~ \hfill $\Box$\\\smallskip}

\newtheorem{theorem}{Theorem}[section]
\newtheorem{lemma}[theorem]{Lemma}

\newtheorem{example}[theorem]{Example}

\newcommand{\Q}{\mathbb{Q}}
\newcommand{\C}{\mathbb{C}}

\newcommand{\R}{\mathbb{R}}

\DeclareMathOperator{\image}{im}

\DeclareMathOperator{\id}{Id}
\DeclareMathOperator{\aut}{Aut}
\DeclareMathOperator{\Diff}{Diff}

\title[Invariant metrics on quasitoric manifolds]{Moduli spaces of invariant metrics of positive scalar curvature on quasitoric manifolds}

\author{Michael Wiemeler}
\address{
Mathematisches Institut\\
WWU M\"unster\\
Einsteinstra\ss{}e 62\\
D-48149 M\"unster\\
Germany}
\email{wiemelerm@uni-muenster.de}
\thanks{The research for this paper was supported by DFG grant HA 3160/6-1.
}

\begin{document}

\begin{abstract}
  We show that the higher homotopy groups of the moduli space of torus-invariant positive scalar curvature metrics on certain quasitoric manifolds are non-trivial.
\end{abstract}

\maketitle


\section{Introduction}

In recent years a lot of work was devoted to the study of the homotopy groups of
the space of Riemannian metrics of positive scalar curvature on a
given closed, connected manifold and its moduli space, see for example the papers \cite{MR2680210}, \cite{MR3270591}, \cite{botvinnik14:_infin}, \cite{MR3268776}, \cite{MR3043028}, \cite{MR2789750}, \cite{wraith16:_non} and the book \cite{MR3445334}.
As far as the moduli space is concerned these results are usually only
for the so-called observer moduli space of positive scalar curvature
metrics, not for the naive moduli space.

The definition of the naive and the observer moduli space are as follows.
The diffeomorphism group of a manifold \(M\) acts by pullback on the space of metrics of positive scalar curvature on \(M\).
The naive moduli space of metrics of positive scalar curvature on \(M\) is the orbit space of this action.

The observer moduli space of metrics is the orbit space of the action of a certain subgroup of the diffeomorphism group, the so-called observer diffeomorphism group.
It consists out of those diffeomorphisms \(\varphi\), which fix some point \(x_0\in M\) (independent of \(\varphi\)) and whose differential \(D_{x_0}\varphi:T_{x_0}M\rightarrow T_{x_0}M\) at \(x_0\) is the identity.

This group does not contain any compact Lie subgroup and therefore acts freely on the space of metrics on \(M\).
Hence, the observer moduli space can be treated from a homotopy theoretic view point more easily than the naive moduli space.

In this paper we deal with the equivariant version of the above
problem: We assume that there is a torus \(T\) acting effectively on the
manifold and that all our metrics are invariant under this torus
action.
To be more specific we study invariant metrics on so-called torus manifolds and quasitoric manifolds.

A torus manifold is a \(2n\)-dimensional manifold with a smooth effective action of an \(n\)-dimensional torus such that there are torus fixed points in the manifold.
Here an action of a torus \(T\) is called effective, if the intersection of the isotropy groups of all points in the manifold is the trivial group.

Such a manifold is called locally standard if it is locally weakly equivariantly diffeomorphic to the standard representation  of \(T=(S^1)^n\) on \(\C^n\), i.e. for each orbit \(Tx\subset M\) there is a neighborhood \(Tx\subset U\subset M\), a diffeomorphism \(\varphi: U\rightarrow V\subset \C^n\) and an automorphism \(\psi\) of \(T\) such that for each \(t\in T\) and \(y\in U\):
\begin{equation*}
  \varphi(ty)=\psi(t)\varphi(y).
\end{equation*}
If \(M\) is locally standard, then the orbit space of the \(T\)-action on \(M\) is naturally a manifold with corners.
\(M\) is called quasitoric if it is locally standard and \(M/T\) is diffeomorphic to a simple convex polytope.

In this paper we use the following notations:
Let \(M\) be a compact manifold.
For a compact connected Lie subgroup \(G\) of \(\Diff(M)\) we denote by
\begin{itemize}
\item \(\mathcal{R}(M,G)\) the space of \(G\)-invariant metrics on \(M\)
\item \(\mathcal{R}^+(M,G)\) the space of \(G\)-invariant metrics of positive scalar curvature on \(M\).
\item \(\mathcal{D}(M,G)=N_{\Diff(M)}(G)/G\) the normalizer of \(G\) in \(\Diff(M)\) modulo \(G\).
\item \(\mathcal{M}(M,G)=\mathcal{R}(M,G)/\mathcal{D}(M,G)\), here the action of \(\mathcal{D}(M,G)\) on \(\mathcal{R}(M,G)\) is given by pullbacks of metrics.
\item \(\mathcal{M}^+(M,G)=\mathcal{R}^+(M,G)/\mathcal{D}(M,G)\), here the action is the restriction of the action on \(\mathcal{R}(M,G)\) to \(\mathcal{R}^+(M,G)\).
\end{itemize}

We equip all these spaces and groups with the \(C^\infty\)-topology or the quotient topology, respectively.

With this notation our main result is as follows:

\begin{theorem}[{Theorem \ref{sec:pi_km-non-triv}}]
\label{sec:introduction}
  There are quasitoric manifolds \(M\) of dimension \(2n\) such that
  for \(0<k<\frac{n}{6}-7\), \(n\) odd and \(k\equiv 0\mod 4\),
  \(\pi_k(\mathcal{M}^+)\otimes \Q\) is non-trivial, where
  \(\mathcal{M}^+\) is some component of \(\mathcal{M}^+(M;T)\).
\end{theorem}

We also show that if a simple combinatorial condition on the orbit polytope of a quasitoric manifold \(M\) is satisfied, then the above theorem holds for \(M\).
We believe that this condition holds for ``most'' quasitoric manifolds.

Note that \(\mathcal{M}^+(M;T)\) is the analogue of the naive moduli space of metrics of positive scalar curvature in the equivariant situation and not the analogue of the observer moduli space for which so far most results have been proven.


The idea of proof for Theorem \ref{sec:introduction} is similar to the ideas used in \cite{MR2680210}:
We first show the following theorem in Section~\ref{sec:action-dm-t-4} which might be of independent interest.

\begin{theorem}
\label{sec:introduction-1}
  Let \(M\) be a quasitoric manifold. 
Then \(\mathcal{R}(M,T)\) is a classifying space for the family of finite subgroups of \(\mathcal{D}(M,T)\).
Moreover, if \(M\) satisfies the above mentioned combinatorial condition, then \(\mathcal{M}(M,T)\) is a rational model for the classifying space \(B\mathcal{D}(M,T)\).
\end{theorem}

The classifying map of an \(M\)-bundle with structure group \(\mathcal{D}(M,T)\), total space \(E\) and paracompact base \(B\) is then given by \(b\mapsto [g|_{E_b}]\), where \(g\) is any \(T\)-invariant Riemannian metric on \(E\) and \(g|_{E_b}\) denotes the restriction of \(g\) to the fiber over \(b\in B\).
The proof of Theorem~\ref{sec:introduction} is then completed by constructing a non-trivial bundle as above with \(B=S^k\), such that there is a metric on \(E\) whose restriction to any fiber has positive scalar curvature.

We refer the reader to \cite{MR1104531}, \cite{MR3363157} and \cite{MR1897064}  as general references on the notions of toric topology.

I want to thank the anonymous referee for detailed comments which helped to improve the presentation of this paper.
I would also like to thank Oliver Goertsches for several comments on an earlier version of this paper.

\section{The action of $\mathcal{D}(M,T)$ on $\mathcal{R}(M,T)$ for $M$ a torus manifold}
\label{sec:action-dm-t-4}

In this section we describe the action of \(\mathcal{D}(M,T)\) on \(\mathcal{R}(M,T)\) where \(M\) is a torus manifold.
We give sufficient criteria for the rational homotopy groups of \(\mathcal{M}(M,T)\) to be isomorphic to the rational homotopy groups of the classifying space of \(\mathcal{D}(M,T)\). 

\begin{lemma}
\label{sec:action-dm-t-2}
  Let \(M\) be a closed manifold.
  If \(T\) is a maximal torus in \(\Diff(M)\), then the isotropy groups of the natural \(\mathcal{D}(M,T)\)-action on \(\mathcal{R}(M,T)\) are finite.
\end{lemma}
\begin{proof}
  The isotropy group of the \(\mathcal{D}(M,T)\)-action of an element \(g\in\mathcal{R}(M,T)\) is the normalizer \(W\)  of the torus \(T\) in the isometry group \(K\) of \(g\) modulo \(T\).
  Since \(M\) is compact \(K\) is a compact Lie group.
  Moreover, because \(T\) is a maximal torus of \(K\), \(W\) is the Weyl group of \(K\) which is a finite group.
  Therefore the statement follows.
\end{proof}

For each torus manifold \(M\) there is a natural stratification of the orbit space \(M/T\) by the identity components of the isotropy groups of the orbits.
That is, the open strata of \(M/T\) are given by the connected components of
 \[(M/T)_{H}=\{Tx\in M/T;\; (T_x)^0=H\}\]
for connected closed subgroups \(H\) of \(T\).
We call the closure of an open stratum a closed stratum.
The closed strata are naturally ordered by inclusion.
We denote by \(\mathcal{P}\) the poset of closed strata of \(M/T\).

There is a  natural map 
\[\lambda:\mathcal{P}\rightarrow \{\text{closed connected subgroups of } T\}\]
such that \(\lambda(F)=H\) if \(F\) is the closure of a component of \((M/T)_H\).

We sometimes also write \(\lambda(Tx)\) or \(\lambda(x)\) for \(\lambda(F_{Tx})\), where \(x\in M\), \(Tx\in M/T\) and \(F_{Tx}\) is the minimal stratum containing \(Tx\).
Note that by this definition for \(x\in M\), \(\lambda(x)=(T_x)^0\subset T\) is the identity component of the isotropy group of \(x\).
If \(M_1\subset M\) is the preimage of a closed stratum \(F\subset M/T\) under the orbit map, then we define \(\lambda(M_1)=\lambda(F)\).

We call \((\mathcal{P},\lambda)\) the characteristic pair of \(M\).

An automorphism of \((\mathcal{P},\lambda)\) is a pair \((f,g)\) such that \(f\) is an automorphism of the poset \(\mathcal{P}\) and \(g\) is an automorphism of the torus \(T\) so that \(\lambda(f(x))=g(\lambda(x))\) for all \(x\in \mathcal{P}\).
The automorphisms of \((\mathcal{P},\lambda)\) naturally form a group \(\aut(\mathcal{P},\lambda)\).

There is a natural action of \(\mathcal{D}(M,T)\) on \(M/T\) which preserves the above stratification.
This action is given by
\begin{equation*}
  \varphi\cdot (Tx)= T\tilde{\varphi}(x),
\end{equation*}
for an orbit \(Tx\) in \(M\) and an element \(\varphi\in \mathcal{D}(M,T)\) which lifts to \(\tilde{\varphi}\in N_{\Diff(M)}T\).
Moreover, we have \(\lambda(\varphi(Tx))=\tilde{\varphi} \lambda(Tx) \tilde{\varphi}^{-1}\). 
Therefore \(\mathcal{D}(M,T)\) acts by automorphisms on the characteristic pair \((\mathcal{P},\lambda)\).

\begin{lemma}
\label{sec:action-dm-t-1}
  Let \(M\) be a torus manifold. Then there is a finite index subgroup \(\mathcal{G}\) of \(\mathcal{D}(M,T)\) which acts freely on \(\mathcal{R}(M,T)\).
  To be more precise, \(\mathcal{G}\) is the kernel of the natural homomorphism \(\mathcal{D}(M,T)\rightarrow \aut(\mathcal{P}, \lambda)\subset \aut(\mathcal{P})\times \aut(T)\), where \((\mathcal{P}, \lambda)\) is the characteristic pair associated to \(M\).
\end{lemma}
\begin{proof}
  At first we show that \(\aut(\mathcal{P},\lambda)\) is a finite group. To see this note that \(\aut(\mathcal{P})\) is finite because \(\mathcal{P}\) is finite.
Moreover, the natural map \(\aut(\mathcal{P},\lambda)\rightarrow \aut(\mathcal{P})\) has finite kernel, because if \(x \in M^T\), then, by the effectiveness of the action, the \(T\)-representation \(T_xM\) is up to automorphisms of \(T\) the standard representation. Therefore, by the slice theorem, \(Tx\) is contained in exactly \(n\) strata \(F_1,\dots,F_n\) of codimension one such that \(\lambda(F_1),\dots,\lambda(F_n)\) generate \(T\) and each \(\lambda(F_i)\) is isomorphic to the circle group.
These \(\lambda(F_i)\) are preserved by the action of the kernel of \(\aut(\mathcal{P},\lambda)\rightarrow \aut(\mathcal{P})\).
Hence, this kernel is isomorphic to a subgroup of \(\prod_{i=1}^n\aut(\lambda(F_i))=(\mathbb{Z}/2\mathbb{Z})^n\).

  Let \(\mathcal{G}\) be the kernel of the natural map \(\mathcal{D}(M,T)\rightarrow \aut(\mathcal{P},\lambda)\).
  Then \(\mathcal{G}\) has finite index since \(\aut(\mathcal{P},\lambda)\) is a finite group.

  Let \(T\subset H\subset \tilde{\mathcal{G}}\) be a compact Lie group which fixes some metric \(g\in \mathcal{R}(M)\), where \(\tilde{\mathcal{G}}\) is the preimage of \(\mathcal{G}\) in \(N_{\Diff(M)}(T)\).
  Then each element of \(H\) commutes with \(T\) and  fixes every \(x\in M^T\), because the \(T\)-fixed points are isolated by dimension reasons.
  Hence, the differential of the \(H\)-action on \(T_xM\) gives an injective homomorphism \(H\rightarrow O(2n)\).
  Since \(T\) is identified with a maximal torus of \(O(2n)\) under this map, it follows that the centralizer of \(T\) in \(O(2n)\) is \(T\) itself.
  Hence it follows that \(H=T\).
  Therefore \(\mathcal{G}=\tilde{\mathcal{G}}/T\) acts freely on \(\mathcal{R}(M,T)\).
\end{proof}

Now let \(M\) be a quasitoric manifold.
Recall that by locally standardness of the action the orbit space \(M/T\) is a smooth manifold with corners, which we require to be diffeomorphic to a simple convex polytope \(P\).
 We denote by \(\pi:M\rightarrow P\) the orbit map.

Similarly to the automorphism group of the characteristic pair \((\mathcal{P},\lambda)\) we define the group \(\Diff(M/T,\lambda)\subset \Diff(M/T)\times \aut(T)\) of those pairs \((f,g)\in \Diff(M/T)\times \aut(T)\) with \(\lambda(f(x))=g(\lambda(x))\) for all \(x\in M/T\).
Here \(\Diff(M/T)\cong \Diff(P)\) denotes the group of all diffeomorphisms of \(M/T\cong P\). Here diffeomorphisms of \(M/T\) are to be understood in the sense of smooth manifolds with corners.
Then we have the following lemma.

\begin{lemma}
\label{sec:action-dm-t}
  For \(M\) a quasitoric manifold, the group \(\mathcal{D}(M,T)\) is naturally isomorphic to  \((C^{\infty}(M/T,T)/T)\rtimes \Diff(M/T,\lambda)\) as topological groups.

In particular the group \(\mathcal{G}\) of Lemma \ref{sec:action-dm-t-1}  is homotopy equivalent to the subgroup of all those diffeomorphisms of \(M/T\) which have the property to map each face of \(M/T\) to itself.
\end{lemma}
\begin{proof}
  First we show that the kernel of the natural map \(\varphi:\mathcal{D}(M,T)\rightarrow \Diff(M/T,\lambda)\) is isomorphic to \(C^{\infty}(M/T,T)/T\).
  Since \(T\) is abelian, there is a natural map from \(C^{\infty}(M/T,T)/T\) to the kernel of \(\varphi\) which is induced by the map \(C^{\infty}(M/T,T)\rightarrow N_{\Diff(M)}(T)\) with \(f\mapsto F\), where \(F(x)=f(Tx)x\) for \(x\in M\).

We show that this map is a homeomorphism.
To do so, let \(\tilde{F}\in N_{\Diff(M)}(T)\), such that \(F=[\tilde{F}]\in \ker\varphi\subset \mathcal{D}(M,T)\).
Then \(\tilde{F}\) maps each orbit in \(M\) to itself.
Since \(M\) is quasitoric, there is a covering of \(M\) by open invariant subsets \(U_1,\dots,U_k\) which are weakly equivariantly diffeomorphic to \(\C^n\) with the standard \(T\)-action.
Here the standard action of \(T=(S^1)^n\) on \(\C^n\) is given by componentwise multiplication.

Because \(\tilde{F}\) maps each \(T\)-orbit to itself, the restriction of \(\tilde{F}\) to \(U_j\cong \C^n\) is of the form
\begin{align*}
  \tilde{F}(z_1,\dots,z_n)&=(f_1(z_1,\dots,z_n),\dots,f_n(z_1,\dots,z_n))\cdot (z_1,\dots,z_n)\\ &=(z_1f_1(z_1,\dots,z_n),\dots,z_nf_n(z_1,\dots,z_n)),
\end{align*}
where \((z_1,\dots,z_n)\in \C^n\) and \(f_k(z_1,\dots,z_n)\in S^1\) for \(k=1,\dots,n\).
Because \(\tilde{F}\) is also \(T\)-equivariant, for each \(k\), \(f_k(z_1,\dots,z_n)\)
 depends only on the orbit of \((z_1,\dots,z_n)\), i.e. on \((|z_1|^2,\dots,|z_n|^2)\).

We have to show that \(f_k\) is smooth for all \(k\).

Smoothness in points with \(z_k\neq 0\) follows from the smoothness of \(\tilde{F}\).
We show that \(f_k\) is also smooth in points with \(z_k=0\).

Since \(\tilde{F}\) is smooth, by the fundamental theorem of calculus, we have for \((z_1,\dots,z_n)\in \C^n\),
\begin{equation*}
  z_kf_k(z_1,\dots,z_n)=\tilde{F}_k(z_1,\dots,z_n)=\int_0^1(D_{z_k}\tilde{F}_k(z_1,\dots,z_{k-1},z_kt,z_{k+1},\dots,z_n))(z_k) \; dt,
\end{equation*}
where
\begin{equation*}
  (D_{z_k}\tilde{F}_k(z_1,\dots,z_n))(z)=\left(\frac{\partial \tilde{F}_k}{\partial x_k}(z_1,\dots,z_n),\frac{\partial \tilde{F}_k}{\partial y_k}(z_1,\dots,z_n)\right)(x,y)^t
\end{equation*}
with \(z_l=x_l+iy_l\) for \(l=1,\dots,n\) and \(z=x+iy\), \(x_l,x,y_l,y\in \R\).

Now we have:
\begin{align*}
  z_kf_k(z_1,\dots,z_n)&=\int_0^1 (D_{z_k}\tilde{F}_k(z_1,\dots,z_{k-1},z_kt,z_{k+1},\dots,z_n))(z_k) \; dt\\
&=\int_0^1z_k (D_{z_k}\tilde{F}_k(z_1,\dots,z_{k-1},\frac{|z_k|}{z_k} z_kt,z_{k+1},\dots,z_n))(1) \; dt\\
&=z_k\int_0^1 (D_{z_k}\tilde{F}_k(z_1,\dots,z_{k-1},|z_k|t,z_{k+1},\dots,z_n))(1) \; dt.
\end{align*}
Here we have used in the second equality that \(\tilde{F}\) is \(T\)-equivariant.

Since \(\tilde{F}\) is \(T\)-equivariant, it follows that 
\begin{equation*}
  t\mapsto \tilde{F}_k(z_1,\dots,z_{k-1},t,z_{k+1},\dots,z_n),\;\; t\in \R
\end{equation*}
is an odd function.
Hence, its \(t\)-derivative
\[t\mapsto
(D_{z_k}F_k(z_1,\dots,z_{k-1},t,z_{k+1},\dots,z_n))(1)\] is an even function.
Therefore the integrand in the last integral depends smoothly on \((z_1,\dots,z_n)\) and \(f_k\) is smooth everywhere.
Because \(f_k\) is \(T\)-invariant, it induces a smooth map on the orbit space, whose derivatives depend continuously on the derivatives of \(\tilde{F}\).

  Hence it is sufficient to show that there is a section to \(\varphi\).

  There is a model \(M\cong ((M/T)\times T)/\sim\), where \((x,t)\sim(x',t')\) if and only if \(x=x'\) and \(t't^{-1}\in \lambda(x)\), which is called canonical model in the literature.
But note that it is not canonical in the sense that it does not depend on any choices.
Actually it depends on a choice of a section to the orbit map. 
  Therefore every \((f,g)\in\Diff(M/T,\lambda)\) of \(M/T\) lifts to a homeomorphism of \(M\) given by \(f\times g\). 
  One can show (see \cite[Lemma 2.3]{MR3080806}), that this homeomorphism is actually a diffeomorphism (for the right choice of the section to the orbit map).
  Therefore we have a section of \(\varphi\) and the first statement follows.

  The second statement follows because \(\mathcal{H}=C^{\infty}(M/T,T)/T\) is contractible because \(M/T\) is contractible.
\end{proof}

By Lemma \ref{sec:action-dm-t}, \(\mathcal{G}/\mathcal{H}\) can be identified with a subgroup of the group of \(T\)-equivariant diffeomorphisms of \(M\).
We fix this identification for the rest of this paper.

\begin{lemma}
\label{sec:action-dm-t-3}
  If in the situation of Lemma~\ref{sec:action-dm-t-1}, \(M\) is quasitoric and the natural homomorphism \(\aut(\mathcal{P},\lambda)\rightarrow \aut(\mathcal{P})\) is trivial, then \(\pi_k(\mathcal{M}(M,T))\otimes \Q \cong \pi_k(B\mathcal{D}(M,T))\otimes \Q\) for \(k>1\).
\end{lemma}
\begin{proof}
  Since \(\mathcal{G}\) acts freely and properly on \(\mathcal{R}(M,T)\), it follows from Ebin's slice theorem \cite{MR0267604} (see also \cite{MR0418147}) that \(\mathcal{R}(M,T)\rightarrow \mathcal{R}(M,T)/\mathcal{G}\) is a locally trivial fiber bundle.
Because \(\mathcal{R}(M,T)\) is contractible,
  \(\mathcal{R}(M,T)/\mathcal{G}\) is weakly homotopy equivalent to \(B\mathcal{G}\).

  Let \(\mathcal{H}=C^{\infty}(M/T,T)/T\) be as in the proof of the previous lemma.
  Then \(\mathcal{H}\) is contractible.
  Hence it follows that \(\mathcal{R}(M,T)\) and \(\mathcal{R}(M,T)/\mathcal{H}\) are weakly homotopy equivalent.

  It follows from Ebin's slice theorem that all \(\mathcal{H}\)-orbits in \(\mathcal{R}(M,T)\) are closed.
Since there is a \(\mathcal{D}(M,T)\)-invariant metric on \(\mathcal{R}(M,T)\), it follows that \(\mathcal{R}(M,T)/\mathcal{H}\) is metriziable.
Hence, \(\mathcal{R}(M,T)/\mathcal{H}\) is paracompact and completely regular.

The \(\mathcal{D}(M,T)\)-invariant metric on \(\mathcal{R}(M,T)\) can be constructed as follows.
Ebin constructs in his paper a sequence of Hilbert manifolds \(\mathcal{R}^s\), \(s\in\mathbb{N}\), such that \(\mathcal{R}(M,\{\id\})=\bigcap_{s\in\mathbb{N}} \mathcal{R}^s\).
On each \(\mathcal{R}^s\) he constructs a \(\Diff(M)\)-invariant Riemannian structure.
This structure induces a \(\Diff(M)\)-invariant metric \(d^s\) on \(\mathcal{R}^s\).
The restrictions of all these metrics \(d^s\) to \(\mathcal{R}(M,\{\id\})\) together induce the \(C^\infty\)-topology on \(\mathcal{R}(M,\{\id\})\).
Therefore the metric
\begin{equation*}
  d(x,y)=\sum_{s\in \mathbb{N}} \min\{d^s(x,y),2^{-s}\}
\end{equation*}
is \(\Diff(M)\)-invariant and induces the \(C^{\infty}\)-topology on \(\mathcal{R}(M,\{\id\})\).

  Since \(\aut(\mathcal{P},\lambda)\rightarrow \aut(\mathcal{P})\) is trivial,
an element \(\tau=(g,f)\in \aut(\mathcal{P},\lambda)\) is of the form \(\tau=(g,f)=(\id,f)\).
Hence, there is a splitting \(\psi:\aut(\mathcal{P},\lambda)\rightarrow \mathcal{D}(M,T)\).
Here \(\tau\) acts on the canonical model\(((M/T)\times T)/\sim\) of \(M\) as identity on the first factor and by \(f\in\aut(T)\) on the second.
To see that this is a diffeomorphism of \(M\), we note that there are
invariant charts \(U\subset M\) which are weakly equivariantly diffeomorphic to \(\C^n\) such that \(U\cap((M/T)\times (\mathbb{Z}_2)^n))/\sim\) is mapped to \(\R^n\subset \C^n\).
For a construction of such charts see \cite[Section 2]{MR3080806}.
The action of \(\tau\) in this chart is given by complex conjugation on some of the factors of \(\C^n\). 

Note that \(\mathcal{H}\) and \(\mathcal{G}\) are normalized by \(\mathcal{K}=\image \psi\).
Moreover, \(\mathcal{K}\) commutes with \(\mathcal{G}/\mathcal{H}\) in \(\mathcal{D}(M,T)/\mathcal{H}\).

  Since \(H=\langle T,\image \psi\rangle\) is a compact Lie subgroup of \(\Diff(M)\),
  there is an \(H\)-invariant metric on \(M\).

  Therefore it follows from \cite[Chapter II.6]{MR0413144} that \[(\mathcal{R}(M,T)/\mathcal{H})/\mathcal{K}=\mathcal{R}(M,T)/(\mathcal{H}\rtimes \mathcal{K})\] is simply connected.
  Moreover, by \cite[Theorem III.7.2]{MR0413144}, one sees that the rational homology of \((\mathcal{R}(M,T)/\mathcal{H})/\mathcal{K}\) vanishes in positive degrees.

  Hence, by the Whitehead theorem, all rational homotopy groups of \(\mathcal{R}(M,T)/(\mathcal{H}\rtimes \mathcal{K})\) vanish.

  By Lemma \ref{sec:action-dm-t-1}, we know that the identity components of \(\mathcal{G}\) and \(\mathcal{D}(M,T)\) are the same.
  Therefore the higher homotopy groups of \(B\mathcal{G}\) and \(B\mathcal{D}(M,T)\) are naturally isomorphic.
  Therefore, by Lemma \ref{sec:action-dm-t} and Ebin's slice theorem, it now suffices to show that \(\mathcal{G}/\mathcal{H}\) acts freely on \(\mathcal{R}(M,T)/(\mathcal{H}\rtimes \mathcal{K})\).

Let \(g\in \mathcal{R}(M,T)\), \(h_1\in \mathcal{G}\), \(h_2\in \mathcal{H}\) such that \(h_1 g= \tau h_2 g\) with \(\tau\in \mathcal{K}\).
Then we have
\begin{equation*}
  \tau^{-1}h_1 g=\tau^{-1} \tau h_2g =h_2g. 
\end{equation*}
Since, by  Lemma \ref{sec:action-dm-t-2}, the isotropy group of \(g\) in \(\mathcal{D}(M,T)\) is finite, it follows that \(\tau^{-1}h_1\) has finite order in \(\mathcal{D}(M,T)/\mathcal{H}\).

Since \(\tau\) acts as the identity on the orbit space, it follows that \(\tau^{-1}h_1\) and \(h_1\) induce the same diffeomorphism on the orbit space.
In particular, \(h_1\) induces a diffeomorphism of finite order \(m\) on \(M/T\)
or equivalently an action of \(\mathbb{Z}/m\mathbb{Z}\) on \(M/T\).
Note that \(h_1\) maps each face of \(M/T\) to itself.

By the slice theorem, for an action of a compact abelian Lie group \(G\) on a connected manifold \(M\) (with or without boundary), there is a unique minimal isotropy group, i.e. a subgroup of \(G\) which is an isotropy group of some orbit in \(M\) and is contained in all other isotropy subgroups.
This subgroup is usually called the principal isotropy subgroup of the action.
The points in \(M\) whose isotropy group is equal to the principal isotropy group form an dense open subset of \(M\).
By the equivariant collaring thoerem, the principal isotropy group of a \(G\)-action on a manifold with boundary is equal to the principal isotropy group of the restricted action on the boundary.
Hence, it follows by induction on the dimension of the faces of \(M/T\) that the diffeomorphism induced by \(h_1\) on \(M/T\) is trivial.
This means that \(h_1\) is contained in \(\mathcal{H}\) and the lemma is proved.
\end{proof}

The proof of the above lemma also shows that the bundle \(\mathcal{G}/\mathcal{H}\rightarrow \mathcal{R}(M,T)/(\mathcal{H}\rtimes\mathcal{K})\rightarrow \mathcal{M}(M,T)\) is rationally a classifying bundle for principal \(\mathcal{G}/\mathcal{H}\)-bundles.
We shall describe the classifying map for bundles \(M\rightarrow E\rightarrow B\)  with structure group \(\mathcal{G}/\mathcal{H}\), fiber \(M\) and paracompact base \(B\).
Since \(\mathcal{G}/\mathcal{H}\) acts on \(M\) by \(T\)-equivariant diffeomorphisms, the \(T\)-action on each fiber extends to an \(T\)-action on \(E\).
Hence \(E\) is a \(T\)-space such that each fiber is \(T\)-invariant and, by paracompactness of \(B\), we may choose a fiberwise \(T\)-invariant Riemannian metric \(g\) on \(E\).
If \(E=B\times M\) is trivial, we therefore have a map 
\begin{align*}
  E=B\times M&\rightarrow \mathcal{R}(M,T)/(\mathcal{H}\rtimes\mathcal{K}) \times M& (b,x)&\mapsto ([g|_{E_b}],x),
\end{align*}
where \(g|_{E_b}\) denotes the restriction of \(g\) to the fiber of \(E\) over \(b\in B\).

If \(E\) is only locally trivial, we still get a map
\begin{equation*}
  E\rightarrow (\mathcal{R}(M,T)/(\mathcal{H}\rtimes\mathcal{K}) \times_{\mathcal{G}/\mathcal{H}} M
\end{equation*}
where on the right-hand side we take the quotient of the diagonal \(\mathcal{G}/\mathcal{H}\)-action.
This map makes the following diagram into a pull-back square,

\begin{equation*}
  \xymatrix{ E \ar[r]\ar[d] & (\mathcal{R}/(\mathcal{H}\rtimes \mathcal{K})) \times_{\mathcal{G}/\mathcal{H}} M \ar[d]\\ B\ar[r]& \mathcal{M}(M,T)},
\end{equation*}

where the bottom map, given by \(b\mapsto [g|_{E_b}]\), is the composition of the classifying
map with the map \(\varphi\)  from the classifying space \(B\mathcal{G}/\mathcal{H}=\mathcal{R}(M,T)/\mathcal{G}\) to \(\mathcal{M}(M,T)=\mathcal{R}(M,T)/\mathcal{D}(M,T)\).
By Lemma \ref{sec:action-dm-t-3}, the map \(\varphi\) is a rational equivalence.

Now we can prove Theorem~\ref{sec:introduction-1} from the introduction.

\begin{proofof}
  The second statement follows from Lemma~\ref{sec:action-dm-t-3} and the above remarks about the classifying maps.

  For the proof of the first statement we fix some notations. Let \(\mathfrak{F}\) be the family of finite subgroups of \(\mathcal{D}(M,T)\) and \(X\) a \(\mathcal{D}(M,T)\)-space. We assume that \(X\) is \(\mathfrak{F}\)-numberable, that is there exists an open covering \(\{U_j;\;j\in J\}\) of \(X\) by \(\mathcal{D}(M,T)\)-subspaces such that:
  \begin{enumerate}
  \item For each \(j\in J\) there exists an equivariant map \(U_J\rightarrow \mathcal{D}(M,T)/H\) with \(H\in \mathfrak{F}\).
  \item There exists a locally finite partion of unity \((t_j;\;j\in J)\) subordinate to \(\{U_j;\;j\in J\}\) such that each \(t_j:X\rightarrow [0,1]\) is a \(\mathcal{D}(M,T)\)-invariant function. 
  \end{enumerate}

We have to show that there exists an equivariant map \(X\rightarrow \mathcal{R}(M,T)\) which is unique up to \(\mathcal{D}(M,T)\)-homotopy.

Since  each compact Lie subgroup of \(N_{\Diff(M)}T\) is the isometry groups of some metric on \(M\) and Ebin's slice theorem, we have equivariant embeddings of \(\mathcal{D}(M,T)/H\) into \(\mathcal{R}(M,T)\) for each finite \(H\).
Hence for each \(j\in J\) we have equivariant maps \(U_j\rightarrow \mathcal{R}(M,T)\).
Using the partion of unity we can convex combine these maps to get an equivariant map \(X\rightarrow \mathcal{R}(M,T)\).

The uniquesness of the map up to homotopy also follows from the convexity of \(\mathcal{R}(M,T)\).
\end{proofof}

\begin{example}
\label{sec:notations}
  We give an example of quasitoric manifolds satisfying the assumptions of the previous lemma.
  
  Let \(n>3\) and \(M_0\) be the projectivization of a sum of \(n-1\) complex line bundles \(E_0,\dots,E_{n-2}\) over \(\C P^1\), such that \(c_1(E_0)=0\) and the first Chern classes of the other bundles are non-trivial, not equal to one and pairwise distinct.
  Then \(M_0\) is a generalized Bott manifold and in particular a quasitoric manifold over \(I\times \Delta^{n-2}\), where \(I\) is the interval and \(\Delta^{n-2}\) denotes an \(n-2\)-dimensional simplex.
For a general description of the combinatorics of the orbit space of a generalized Bott manifold see for example \cite[Section 6]{MR2666127}.

  Let \(M_1=\C P^1\times M_0\) and \(M_2=M_1\#\overline{\C P^n}\) the blow up of \(M_1\) at a single \(T\)-fixed point. The orbit space of \(M_1\) is \(I\times I\times \Delta^{n-2}\). The orbit space of \(M_2\) is the orbit space of \(M_1\) with a vertex cut off, i.e. \(M_2/T=(M_1/T)\# \Delta^n\), where the connected sum is taken at a vertex.

The combinatorial types of the facets of \(M_2/T\) are given as in table \ref{tab:1} below.
Since the combinatorial types of facets in the lines in this table are pairwise distinct, it follows that the lines in the table are invariant under the action of \(\aut(\mathcal{P},\lambda)\).
Therefore the facets in the first two lines are fixed by the action of this group.
The facets in lines 3 and 4 are fixed, because in each of these lines there appears one facet \(F\) with \(\lambda(F)=\{(z,1,\dots,1)\in T^n;\; z\in S^1\}\) but the values of \(\lambda\) on the other facets are distinct.

Finally the facets \(F_1,\dots,F_{n-2}\) in the last line are fixed, by all \((f,g)\in\aut(\mathcal{P},\lambda)\) because \(g\) must permute the subgroups \(\lambda(F_1),\dots,\lambda(F_{n-2})\), which are the coordinate subgroups in \(\{(1,1)\}\times (S^1)^{n-2}\), and must also fix the subgroups \(\lambda(F')\) with \(F'\) from line 3.

Note that depending on the choices of the bundles \(E_0,\dots,E_{n-2}\), \(M_2\) can be spin or non-spin.
\end{example}

\begin{table}
  \centering
  \begin{tabular}{|c|c|l|}
    & combinatorial type & \((\alpha_1,\dots,\alpha_n)\)\\\hline\hline
    \(1\)& \(\Delta^{n-1}\)& \((1,\dots,1)\)\\\hline
    \(1\)& \(I\times I\times \Delta^{n-3}\)& \((0,0,1,\dots,1)\)\\\hline
    \(2\)& \(I\times \Delta^{n-2}\)& \((1,0,0,\dots,0)\)\\
         &                        & \((0,1,k_1,\dots,k_{n-2})\)\\
&& with \(k_i\) pairwise distinct and non-zero\\\hline 
    \(2\)& \(I\times \Delta^{n-2}\) with vertex cut off& \((1,0,0,\dots,0)\)\\
         &                                            & \((0,1,0,\dots,0)\)\\\hline
    \(n-2\) & \(I\times I \times \Delta^{n-3}\) with vertex cut off & \((0,0,0,\dots,0,1,0,\dots,0)\)
  \end{tabular}
  \caption{The combinatorial types of the facets of \(M_2/T\). In the first column the numbers of facets of these type are given. In the last column the values of \(\lambda(F)=\{(z^{\alpha_1},\dots,z^{\alpha_n})\in T^n;\; z\in S^1\}\) are given.}
  \label{tab:1}
\end{table}

\section{The homotopy groups of $\mathcal{D}(M,T)$ for $M$ a quasitoric manifold}

In this section we show that, for some quasitoric manifolds \(M\) of dimension \(2n\), \(n\) odd, the rational homotopy groups of \(\mathcal{D}(M,T)\) are non-trivial in certain degrees.
Let \(P\) be the orbit polytope of \(M\).

Let \(D^n\hookrightarrow P\) be an embedding into the interior of \(P\) such that \(K=P-D^n\) is a collar of \(P\).
Then we have a decomposition
\begin{equation*}
  M=(D^n\times T)\cup_{S^{n-1}\times T} \pi^{-1}(K)=(D^n\times T)\cup_{S^{n-1}\times T} N.
\end{equation*}

From this decomposition we get a homomorphism \(\psi:\widetilde{\Diff}(D^n,\partial D^n)\rightarrow \mathcal{G}/\mathcal{H}\hookrightarrow \mathcal{D}(M,T)\) by letting a diffeomorphism of \(D^n\) act on \(M\) in the natural way on \(D^n\) and by the identity on \(T\) and \(N\).
Here \(\widetilde{\Diff}(D^n,\partial D^n)\) denotes the group of those diffeomorphisms of \(D^n\) which are the identity on some collar neighborhood of the boundary.
By the uniqueness of collars up to isotopy, it is weakly homotopy equivalent to the group \(\Diff(D^n,\partial D^n)\) of all diffeomorphisms of \(D^n\) which are the identity on the boundary.

There is also a natural map \(\widetilde{\Diff}(D^n,\partial D^n)\rightarrow \Diff(P)\), because a diffeomorphism in \(\widetilde{\Diff}(D^n,\partial D^n)\) can be extended by the identity on \(K\) to form a diffeomorphism of \(P\).
This natural map factors as \(\pi_*\circ \psi\), where \(\pi_*:\mathcal{D}(M,T)\rightarrow \Diff(P)\) is the natural map induced by the orbit map.

\begin{lemma}
  For \(0<k<\frac{n}{6}-8\), \(n\) odd and \(k\equiv -1 \mod 4\), the natural map
  \begin{equation*}
    \pi_k(\Diff(D^n,\partial D^n))\otimes \Q\cong\pi_k(\widetilde{\Diff}(D^n,\partial D^n))\otimes \Q\rightarrow \pi_k(\Diff(P))\otimes \Q
  \end{equation*}
is injective and non-trivial. In particular \(\psi\) induces an injective non-trivial homomorphism on these homotopy groups.
\end{lemma}
\begin{proof}
  We have exact sequences
  \begin{equation*}
    1\rightarrow \Diff(D^n,\partial D^n)\rightarrow \widetilde{\Diff}(P) \rightarrow \Diff(K),
  \end{equation*}
where \(\widetilde{\Diff}(P)\) is the group of diffeomorphisms of \(P\) which preserve \(K\),
and
\begin{equation*}
  1\rightarrow \Diff(K,\partial D^n)\rightarrow \Diff(K)\rightarrow \Diff(\partial D^n).
\end{equation*}

Note that \(\widetilde{\Diff}(P)\) is weakly homotopy equivalent to \(\Diff(P)\), by the uniqueness of collars up to isotopy.
Moreover, the images of the right-hand maps in the above sequences have finite index as we explain now.

In the first sequence this is because the group of those diffeomorphisms of a sphere which extend to diffeomorphisms of the disc has finite index in all diffeomorphisms of the sphere.

To see that \(\Diff(K)\rightarrow \Diff(\partial D^n)\) is surjective, we have to show that every diffeomorphism of \(\partial D^n\) extends to a diffeomorphism of \(K\).
This can be done as in the last step of the proof of Theorem 5.1 of \cite{MR3030690}.



Therefore we get exact sequences of rational homotopy groups
\begin{equation*}
  \pi_{k+1}(\Diff(P))\otimes \Q\rightarrow \pi_{k+1}(\Diff(K))\otimes \Q\rightarrow \pi_k(\Diff(D^n,\partial D^n))\otimes \Q\rightarrow \pi_k(\Diff(P))\otimes \Q
\end{equation*}
and
\begin{equation*}
  \pi_{k+1}(\Diff(K,\partial D^n))\otimes \Q\rightarrow \pi_{k+1}(\Diff(K))\otimes \Q\rightarrow \pi_{k+1}(\Diff(\partial D^n))\otimes \Q.
\end{equation*}
By Farrell and Hsiang \cite{MR520509}, we have \(\pi_{k+1}(\Diff(\partial D^n))\otimes \Q=0\).

Moreover every family in the
image of \(\pi_{k+1}(\Diff(K,\partial D^n))\rightarrow \pi_{k+1}(\Diff(K))\) extends to a family of diffeomorphisms of \(P\), by defining the extension to be the identity on \(D^n\).

Therefore the map \(\pi_{k+1}(\Diff(K))\otimes \Q\rightarrow
\pi_k(\Diff(D^n, \partial D^n)\otimes \Q\) is the zero map and the
claim follows from Farrell and Hsiang \cite{MR520509}.
\end{proof}

\section{$\pi_k(\mathcal{M}^+)$ is non-trivial}

In this section we show that \(\pi_k(\mathcal{M}^+(M,T))\) is non-trivial for manifolds as in Example \ref{sec:notations}.

To do so, we need the following theorem which is an equivariant version of Theorem 2.13 of \cite{MR2789750}.

\begin{theorem}
  \label{sec:pi_km-non-triv-1}
  Let \(G\) be a compact Lie group. Let \(X\) be a smooth compact \(G\)-manifold of dimension \(n\) and \(B\) a
compact space. Let \(\{g_b \in \mathcal{R}^+(X,G) : b \in B\}\) be a continuous family of invariant metrics of positive scalar curvature.
Moreover, let \(\iota: G\times_H(S(V)\times D_1(W))\rightarrow X\) be an equivariant embedding, with \(H\subset G\) compact, \(V,W\) orthogonal \(H\)-representations with \(\dim G - \dim H +\dim V +\dim W=n+1\) and \(\dim W > 2\).
Here \(S(V)\) and \(D_1(W)\) denote the unit sphere and the unit disc in \(V\) and \(W\), respectively.

Finally let \(g_{G/H}\) be any \(G\)-invariant metric on \(G/H\) and \(g_V\) be any \(H\)-invariant metric on \(S(V)\). 

 Then, for some \(1>\delta>0\), there is a continuous
map
\begin{align*}
  B&\rightarrow \mathcal{R}^+(X,G)\\
b&\mapsto g^b_{std}
\end{align*}
satisfying

\begin{enumerate}
\item Each metric \(g_{std}^b\)
makes the map \(G\times_H(S(V)\times D_\delta(W))\rightarrow (G/H,g_{G/H})\) into a Riemannian submersion. Each fiber of this map is isometric to \((S(V)\times D_\delta(W),g_V+g_{tor})\), where \(g_{tor}\) denotes a torpedo metric on \(D_\delta(W)\).
Moreover \(g_{std}^b\) is the original metric outside a slightly bigger neighborhood of \(G\times_H(S(V)\times\{0\})\).
\item The the original map \(B\rightarrow \mathcal{R}^+(X,G)\) is homotopic to the new map.
\end{enumerate}
\end{theorem}

The proof of this theorem is a direct generalization of the proof of Theorem 2.13 of \cite{MR2789750} using the methods of the proof of Theorem 2 in \cite{MR2376283}.
Therefore we leave it to the reader.

Let \(E\) be the total space of a Hatcher disc bundle \cite{goette01:_morse_i}  over \(S^k\) with fiber \(D^n\) and structure group \(\Diff(D^n,\partial D^n)\).
Note that its classifying map \(S^k\rightarrow B \Diff(D^n,\partial D^n)\) represents a non-trivial element in \(\pi_k(B\Diff(D^n,\partial D^n))\).

Moreover, let
\begin{equation*}
  F= (E\times T)\cup_{S^k\times S^{n-1}\times T} (S^k\times N),
\end{equation*}
where 
\[M=(D^n\times T)\cup_{S^{n-1}\times T} \pi^{-1}(K)=(D^n\times T)\cup_{S^{n-1}\times T} N\]
is a \(2n\)-dimensional quasitoric manifold over the polytope \(P\) and \(K\) is a collar of the boundary of \(P\).

Let \(M_1\subset N\) be a characteristic submanifold.
Then \(M_1\) is a submanifold of codimension two in \(N\) which is fixed pointwise by a circle subgroup \(\lambda(M_1)\) of \(T\).
\(M_1\) is the preimage of a facet of \(P\) under the orbit map.
Denote by \(\tilde{M_1}\) an equivariant tubular neighborhood of \(M_1\).

Then \(F\) is a bundle over \(S^k\) with fiber the quasitoric manifold \(M\) and structure group \(\Diff(D^n,\partial D^n)\).
Note that \(F\) has a natural fiberwise \(T\)-action.

By Theorem 2.9 of \cite{MR2680210}, we have a metric on \(E\) with fiberwise positive scalar curvature which is a product metric at the boundary.
On \(T\times D^2\) we choose an \(T\times S^1\)-invariant metric of non-negative scalar curvature on \(T\times D^2\), which is also a product metric at the boundary. Here \(S^1\) acts by rotation on \(D^2\).

On \((N-\tilde{M_1})\times D^2\), there is an equivariant Morse function \(h\) without critical orbits of co-index less than three.
Indeed, in \cite[Proof of Theorem 2.4]{MR3449263} we constructed an equivariant Morse function \(f\) on \(M-\tilde{M_1}\) without handles of co-index zero.
In this construction we can arrange that the global minimum of \(f\) is attained in a principal orbit.
By restricting \(f\) to the complement of an invariant neighborhood of this principal orbit, we get an equivariant Morse function on a \(T\)-manifold with boundary which is equivariantly diffeomorphic to \(N-\tilde{M_1}\).
This function induces a Morse function \(h'\) on \((N-\tilde{M_1})\times D^2\) such that
\begin{equation*}
  h'(x,y)=f(x)+\|y\|^2\;\;(x,y)\in (N-\tilde{M_1})\times D^2
\end{equation*}
We can deform this function \(h'\) in a neighborhood of the boundary of \(N-\tilde{M}_1\) to an equivariant Morse-function \(h\) in such a way that:
\begin{itemize}
\item There are no critical orbits in a neighborhood of the boundary.
\item The global minimum of \(h\) is attained on \((\partial N)\times D^2\).
\item The global maximum of \(h\) is attained on \((\partial \tilde{M}_1)\times D^2\).
\item The critical orbits of \(h\) are contained in \((N-\tilde{M}_1)\times \{0\}\) and all have co-index at least three.
\end{itemize}

Using this Morse function and Theorem~\ref{sec:pi_km-non-triv-1} we get a fiberwise invariant metric of positive scalar curvature on \(\partial ((F-(S^k\times \tilde{M_1}))\times D^2)\).

Indeed, using the function \(h\), we get an equivariant handle decomposition of \((N-\tilde{M_1})\times D^2\), without handles of codimension less than three, i.e.
\[(N-\tilde{M_1})\times D^2= (\partial N)\times D^2\times I \cup T\times_{H_1} (D(V_1)\times D(W_1))\cup\dots\cup T\times_{H_k}(D(V_k)\times D(W_k),\]
such that the \(H_i\) are closed subgroups of \(T\), the \(V_i\) and \(W_i\) are orthogonal \(H_i\) representations with \(\dim W_i\geq 3\) and the gluing of a handle \(T\times_{H_i}(D(V_i)\times D(W_i))\) is performed along \(T\times_{H_i}(S(V_i)\times D(W_i))\).

Moreover, the restriction of the bundle \(\partial(E\times T\times D^2)\rightarrow S^k\) to \((\partial E)\times T \times D^2\) is trivialized by assumption.
So the restriction of \(g\) to the fibers of this bundle gives a compact family of invariant metrics of positive scalar curvature on \((\partial N)\times D^2\).
By Theorem \ref{sec:pi_km-non-triv-1}, we can assume that this family is in standard form on the attaching locus of \(T\times_{H_1}(D(V_1)\times D(W_1))\) in \(\partial N \times D^2\).
Therefore the family of product metrics on \((\partial N) \times D^2\times I\) extends to a family of  invariant positive scalar curvature metrics on 
\begin{equation*}(\partial N) \times D^2\times I\cup T\times_{H_1}(D(V_1)\times D(W_1)),\end{equation*}
which are product metrics at the boundary.
Continuing in the same manner with the other handles leads to 
 a family of invariant metric of positive scalar curvature on \((N-\tilde{M}_1)\times D^2\) which are product metrics at the boundary.
Gluing this metric together with the metric on \(E\times T\times D^2\) and restricting to the boundary leads to an fiberwise invariant metric of positive scalar curvature on \(\partial((F-S^k\times \tilde{M}_1)\times D^2)\).

Note that Berard Bergery's result \cite{berard83:_scalar} on the existence of a metric of positive scalar curvature on the orbit space of a free torus action, generalizes directly to a family version.
This is because Berard Bergery shows that if \(g\) is an invariant metric of positive scalar curvature on a free \(S^1\)-manifold \(M\), then \(f^{2/\dim M - 2}\cdot g^*\) has positive scalar curvature, where \(g^*\) is the quotient metric of \(g\) and \(f\) is the length of the \(S^1\)-orbits in \(M\).
This construction clearly generalizes to families of metrics.
Moreover the metrics on the orbit space will be invariant under every Lie group action which is induced on \(M/S^1\) from an action on \(M\) which commutes with \(S^1\) and leaves the metrics on \(M\) invariant (see \cite[Theorem 2.2]{MR3449263}  for the case of a single metric).

We have a free action of the diagonal in \(\lambda(M_1)\times S^1\cong S^1\times S^1\) on 
\[\partial ((F-(S^k\times \tilde{M_1}))\times D^2)=(S^k \times \partial \tilde{M_1}\times D^2)\cup_{S^k\times(\partial \tilde{M_1})\times S^1}(F-S^k\times \tilde{M}_1)\times S^1.\]
Here the first factor of \(\lambda(M_1)\times S^1\) acts as a subgroup of \(T\) on \(F\) and the second factor acts on \(D^2\) by rotation.

The orbit space of this free action is
\[(S^k \times \tilde{M}_1)\cup_{S^k\times (\partial \tilde{M_1})}(F-S^k\times \tilde{M}_1),\]
which is clearly equivariantly diffeomorphic to \(F\).

Hence, with the remarks from above one gets an invariant metric of fiberwise positive scalar curvature on \(F\) in the same way as in the case of a single metric (see \cite[Proof of Theorem 2.4]{MR3449263} for details). 

This metric defines an element \(\gamma\)  in 
\(\pi_k(\mathcal{M}^+(M,T))\otimes \mathbb{Q}\).
The image of \(\gamma\) in 
\[\pi_k(\mathcal{M}(M,T))\otimes \mathbb{Q}\cong \pi_k(B\mathcal{D}(M,T))\otimes \mathbb{Q}\]
 is represented by the classifying map for our Hatcher bundle \(E\).

Therefore it follows from the lemmas in the previous two sections, that \(\gamma\) is non-trivial if \(M\) is as in Example \ref{sec:notations} because the classifying map of a Hatcher bundle represents a non-trivial element in the homotopy groups of \(B\Diff(D^n,\partial D^n)\).

Therefore we have proved the following theorem:

\begin{theorem}
\label{sec:pi_km-non-triv}
  Let \(M\) be a quasitoric manifold of dimension \(2n\) such that \(\aut(\mathcal{P},\lambda)\rightarrow \aut(\mathcal{P})\) is trivial.
  Then for \(0<k<\frac{n}{6}-7\), \(n\) odd and \(k\equiv 0\mod 4\),
  \(\pi_k(\mathcal{M}^+)\otimes \Q\) is non-trivial, where
  \(\mathcal{M}^+\) is some component of \(\mathcal{M}^+(M;T)\).
\end{theorem}

\bibliography{moduli}
\bibliographystyle{alpha}

\end{document}